\documentclass{article}                

\usepackage{graphicx}
\usepackage{mathptmx}      

\usepackage{lineno,hyperref}
\modulolinenumbers[5]

\usepackage{amsmath}
\usepackage{amstext}
\usepackage{amsfonts}
\usepackage{amssymb}
\usepackage{amsthm}
\usepackage{mathtools}
\usepackage{authblk}
\usepackage{natbib}

\usepackage{titlesec}
\titleformat*{\section}{\large\bfseries}

\newtheoremstyle{MyTheorem}
{\topsep}       
{\topsep}       
{\itshape}  
{}          
{\bfseries}  
{.}         
{5pt plus 1pt minus 1pt}      
{}          

\theoremstyle{MyTheorem}
\newtheorem{theorem}{Theorem}
\newtheorem{cond}[theorem]{Condition}

\newcommand\IPst[0]{\text{s.t.\;\,}}
\newcommand\IPmin[0]{\min\;}

\newcommand\cref[1]{(\ref{#1})}

\begin{document}

\title{Compact Linearization for Binary Quadratic Problems subject to Linear Equations}

\author{Sven Mallach}

\affil{Institut f\"ur Informatik \authorcr Universit\"at zu K\"oln, 50923 K\"oln, Germany}
\date{December 15, 2017}

\maketitle

\begin{abstract}
In this paper it is shown that the compact linearization approach, that has been previously
proposed only for binary quadratic problems with assignment constraints, can be generalized
to arbitrary linear equations with positive coefficients which 
considerably enlarges its applicability. We discuss special cases of prominent quadratic
combinatorial optimization problems where the obtained compact linearization yields
a continuous relaxation that is provably as least as strong as the one obtained with an
ordinary linearization.
\end{abstract}

\section{Introduction}

Since linearizations of quadratic and, more generally, polynomial programming problems, enable
the application of well-studied mixed-integer linear programming techniques, they have been an
active field of research since the 1960s.
The seminal idea to model binary conjunctions using additional (binary) variables
is attributed to~\citet{Fortet59,Fortet60} and addressed by~\citet{HammerRudeanu}.
This method, that is also proposed in succeeding works by~\citet{Balas64}, \citet{Zangwill}
and \citet{Watters}, and further discussed by~\citet{GloverWoolsey73}, requires two
inequalities per linearization variable. Only shortly thereafter, \citet{GloverWoolsey74}
found that the same effect can be achieved using \emph{continuous} linearization
variables when replacing one of these inequalities with two different ones.
The outcome is a method that~is until today regarded as \lq the standard
linearization technique\rq\ and where, in the binary quadratic case, each product
$x_{i} x_{j}$ is modeled using a variable $y_{ij} \in [0,1]$ and three constraints:
\begin{align}
  y_{ij}  &\le\;  x_{i} && \label{pck:stdlin1} \\
  y_{ij}  &\le\;  x_{j} && \label{pck:stdlin2} \\
  y_{ij}  &\ge\;  x_{i} + x_{j} -1 && \label{pck:stdlin3}
\end{align}

While this approach is widely applicable, it is considered to be rather \lq weak\rq\ in
the sense that the inequalities couple the three involved variables only very loosely.
Depending on the concrete problem formulation, this may result in linear programming
relaxations that yield unsatisfactory bounds on the respective objective function.

Several attempts have been made to develop more sophisticated techniques
that exploit the structure of a given problem formulation.
For example, linearization methods for binary quadratic problems (BQPs) where
all bilinear terms appear only in the objective function are proposed by~\citet{Glover75},
\citet{OralKettani92a,OralKettani92b}, as well as by \citet{CPP2004}, \citet{SheraliSmith2007},
\citet{FuriniTraversi}, and, for quadratic problems with general integer
variables, by~\cite{Billionnet2008}.
Adapted formulations for unconstrained BQPs
were addressed by~\citet{GueyeMichelon2009}, and \citet{HansenMeyer2009}.

In this paper, we are concerned with a compact linearization technique for BQPs
that comprise a collection $K$ of linear equations
$\sum_{i \in A_k} a_i^k x_i = b^k$, $k \in K$, where
$A_k$ is an index set specifying the binary variables on the left
hand side, $a_i^k \in \mathbb{R}^{>0}$ for all $i \in A_k$, and ${b^k \in \mathbb{R}^{>0}}$.
For ease of notation, let the overall set of binary variables in the BQP be indexed by
a set $N = \{ 1, \dots, n\}$ where $n \in \mathbb{N}^{>0}$.
Bilinear terms $y_{ij} = x_i x_j$, are permitted to occur in
the objective function as well as in the set of constraints. The technique proposed
is \emph{compact} in the sense that it typically adds less constraints than the
\lq standard\rq\ method. Its only prerequisite is that for each product $x_i x_j$
there exist indices $k, \ell$ such that $i \in A_k$ and ${j \in A_\ell}$, i.e., for
each variable being part of a bilinear term there is some linear equation involving
it. Without loss of generality, we also assume the bilinear terms to be collected in an
ordered set ${P \subset N \times N}$ such that $i \le j$ for each ${(i,j) \in P}$.
With an arbitrary set of $m \ge 0$ linear constraints
$C x + D y \ge e$ where $C \in \mathbb{R}^{m \times n}$ and
${D \in \mathbb{R}^{m \times |E|}}$, a general form of the mixed-integer
programs considered can be stated as: 
\begin{align}
\IPmin  &  c^Tx + d^Ty                       &      &            && \nonumber \\
\IPst   &  \sum_{i \in A_k} a_i^k x_{i}       &=\;   &  b^k       && \mbox{for all } k \in K      \label{bqp:eqn} \\
        &  C x + D y                         &\ge\; &  e         && \nonumber \\
	&  y_{ij}                            &=\;   & x_i x_j    && \mbox{for all } (i, j) \in P \label{bqp:bilinear} \\
	&  x_i                               &\in\; & \{0,1\}    && \mbox{for all } i \in N      \nonumber
\end{align}

This general mixed-integer program covers several NP-hard combinatorial
optimization problems.
For the case where the equations (\ref{bqp:eqn}) are assignment constraints, i.e.,
for all $k \in K$ we have $b^k = 1$ and $a^k_i = 1$ for all $i \in A_k$,
\citet{Liberti2007} and \citet{Mallach2017} proposed a compact linearization
approach that can be seen as a first level application of the so-called
reformulation-linearization technique by~\citet{SheraliAdamsRLT}.

In this paper, we show how this approach can be generalized to the
mentioned arbitrary linear equations with positive coefficients. This
broadens the applicability of the approach in a strong sense. We then
discuss under which circumstances the obtained compact linearization
is provably as least as strong as the one obtained with
the \lq standard linearization\rq. Finally, we highlight some prominent
example applications where previously found linearizations appear as
special cases of the proposed technique.

\section{Compact Linearization}\label{s:CmpLin}
The compact linearization approach for binary quadratic
problems with linear equations is as follows.
With each linear equation of type (\ref{bqp:eqn}), i.e., with each index
set $A_k$, we associate a corresponding index set $B_k \subseteq N$ such
that for each $j \in B_k$ the equation is multiplied with~$x_j$.
We thus obtain the new equations:
\begin{align}
    \sum_{i \in A_k} a_i^k x_i x_j     &=\;  b^k x_{j} && \mbox{for all } j \in B_k,\; \mbox{for all } k \in K \label{cmp:eqn0}
\end{align}

Each product $x_i x_j$ induced by any of the equations~\cref{cmp:eqn0} is then
replaced by a continuous linearization variable $y_{ij}$ (if $i \le j$) or
$y_{ji}$ (otherwise). We denote the set of bilinear terms created this way
with
$$Q = \{ (i,j) \mid i \le j \mbox { and } \exists k \in K: i \in A_k, j \in B_k \mbox{ or } j \in A_k, i \in B_k \}.$$
Rewriting the equations \cref{cmp:eqn0} using $Q$, we obtain the linearization equations:
\begin{align}
    \sum_{i \in A_k, (i,j) \in Q} a_i^k y_{ij}\ + \sum_{i \in A_k, (j,i) \in Q} a_i^k y_{ji}    &=\; b^k x_{j} && \mbox{for all } j \in B_k,\; \mbox{for all } k \in K \label{cmp:eqn1}
\end{align}

It is clear that the equations (\ref{cmp:eqn0}) are valid for
the original problem and thus as well the equations (\ref{cmp:eqn1})
whenever the introduced linearization variables take on consistent
values with respect to their two original counterparts, i.e.,
$y_{ij} = x_i x_j$ holds for all~${(i, j) \in Q}$.

In order to obtain a linearization of the original problem
formulation with the bilinear terms defined by the set $P$,
we need to choose the sets $B_k$ such that the induced set of
variables $Q$ will be equal to or contain $P$ as a subset. We
will discuss how to obtain a set $Q \supseteq P$ later, but
suppose for now that such a set $Q$ is already at hand.
We will show that a consistent
linearization is obtained if and only if the following two
conditions (for which $k = \ell$ is a valid choice) are satisfied:
\begin{cond} \label{cond:c1}
For each $(i,j) \in Q$, there is a $k \in K$
such that $i \in A_k$ and $j \in B_k$.
\end{cond}
\begin{cond} \label{cond:c2}
For each $(i,j) \in Q$, there is an $\ell \in K$
such that $j \in A_\ell$ and $i \in B_\ell$.
\end{cond}

\begin{theorem} \label{thm:main}
If Conditions~\ref{cond:c1} and~\ref{cond:c2} are satisfied, then
for any integer solution $x \in \{0,1\}^N$, the inequalities
$y_{ij} \le x_i$, $y_{ij} \le x_j$ and
$y_{ij} \ge x_i + x_j - 1$ hold for all $(i,j) \in Q$.
\end{theorem}
\begin{proof}
Let $(i,j) \in Q$.
By Condition~\ref{cond:c1}, there is a $k \in K$ such that
$i \in A_k$, $j \in B_k$ and hence the equation
\begin{align}
    \sum_{h \in A_k, (h,j) \in Q} a_h^k y_{hj} + \sum_{h \in A_k, (j,h) \in Q} a_h^k y_{jh}   &=\; b^k x_{j} && \tag{$*$} \label{eqn:rhsj_proof}
\end{align}
exists and has $y_{ij}$ on its left hand side. Since
$a_h^k > 0$ for all $h \in A_k$ and $0 \le y_{ij} \le 1$,
it establishes that $y_{ij} = 0$ whenever $x_j = 0$.
Similarly, by Condition~\ref{cond:c2}, there is an $\ell \in K$ such that
$j \in A_\ell$, $i \in B_\ell$ and hence the equation
\begin{align*}
    \sum_{h \in A_\ell, (h,i) \in Q} a_h^\ell y_{hi} + \sum_{h \in A_\ell, (i,h) \in Q} a_h^\ell y_{ih}   &=\; b^\ell x_{i} &&  \tag{$**$} \label{eqn:rhsi_proof}
\end{align*}
exists and has $y_{ij}$ on its left hand side. Since
$a_h^\ell > 0$ for all $h \in A_\ell$, $0 \le y_{ij} \le 1$
it establishes that $y_{ij} = 0$ whenever $x_i = 0$.

Within a framework that constructs
a linearization only by means of equations of type (\ref{cmp:eqn1}), there
is no other way of establishing $y_{ij} \le x_j$ and $y_{ij} \le x_i$ than
by satisfying conditions~\ref{cond:c1} and~\ref{cond:c2} -- showing their
\emph{necessity}. We will now continue with showing also their \emph{sufficiency}.

Since $y_{ij} = 0$ whenever $x_i = 0$ or $x_j = 0$, the inequality
$y_{ij} \ge x_i + x_j - 1$ is satisfied as well in this case 
So let now $x_i = x_j = 1$. Then the right hand sides of~(\ref{eqn:rhsj_proof})
and~(\ref{eqn:rhsi_proof}) are equal to $b^k$ and $b^\ell$ respectively.
The variable $y_{ij}$ (is the only one that) occurs on the left hand sides
of both of these equations. If $y_{ij} = 1$, this is consistent and correct.
So suppose that $y_{ij} < 1$ which implies that, in the case of
equation~(\ref{eqn:rhsj_proof}), we have:
\begin{align}
    \sum_{h \in A_k, (h,j) \in Q, h \neq i} a_h^k y_{hj} + \sum_{h \in A_k, (j,h) \in Q, h \neq i} a_h^k y_{jh}  &=\; b^k \hspace{-3pt} \underbrace{x_j}_{=1} -\;\; a_i^k \hspace{-1pt} \underbrace{y_{ij}}_{< 1} > b^k - a_i^k. && \tag{$*'$} \label{eqn:rhsj_proof2}
\end{align}

At the same time, however, we have $\sum_{h \in A_k, h \neq i} a_h^k x_h = b^k - a_i^k$ with $x_h \in \{0,1\}$.
In order for the equation~(\ref{eqn:rhsj_proof}) to be satisfied, an additional amount
of $(1 -y_{ij})a_i^k > 0$ must be contributed by the other summands on the left hand side of~(\ref{eqn:rhsj_proof2}).
This implies, however, that there must be some $h \in A_k$, $h \neq i$, such that $y_{hj} > 0$
(or $y_{jh} > 0$) while $x_h = 0$ -- which is impossible since the
conditions~\ref{cond:c1} and~\ref{cond:c2} are also established for these variables.
An analogous result can be stated for equation~(\ref{eqn:rhsi_proof}).
\end{proof}

Unfortunately, in the general case, Theorem~\ref{thm:main} cannot be restated for fractional
solutions $x \in [0,1]^N$ and thus we cannot conclude from the proof
that the compact linearization yields a linear programming relaxation which
is provably as least as tight as the one obtained with the \lq standard linearization\rq.
This is in contrast to the special case where the equations (\ref{bqp:eqn}) are assignment
constraints, i.e., $a^k_i = 1$ for all $i \in A_k$ and $b^k = 1$ for all $k \in K$.
As stated in the introduction, the compact linearization approach was originally proposed
for this case in~\citet{Liberti2007} and~\citet{Mallach2017}. Accidentally, the proof
in~\citet{Mallach2017} lacks a verification that $y_{ij} \ge x_i + x_j - 1$ holds
for all $(i,j) \in Q$ also in the fractional case that we now catch up on.

\begin{theorem} \label{thm:ass}
If, for all $k \in K$, $a^k_i = 1$ for all $i \in A_k$ and $b^k = 1$, and the
Conditions~\ref{cond:c1} and~\ref{cond:c2} are satisfied, then for any
solution $x \in [0,1]^N$, the inequalities
$y_{ij} \le x_i$, $y_{ij} \le x_j$ and
$y_{ij} \ge x_i + x_j - 1$ hold for all $(i,j) \in Q$.
\end{theorem}
\begin{proof}
Let $(i,j) \in Q$.
By Condition~\ref{cond:c1}, there is a $k \in K$ such that
$i \in A_k$, $j \in B_k$ and hence the equation
\begin{align}
    \sum_{h \in A_k, (h,j) \in Q} y_{hj} + \sum_{h \in A_k, (j,h) \in Q} y_{jh}   &=\; x_{j} && \tag{$***$} \label{eqn:rhsj_proof3}
\end{align}
exists, has $y_{ij}$ on its left hand side, and thus establishes $y_{ij} \le x_j$.
Similarly, by Condition~\ref{cond:c2}, there is an $\ell \in K$ such that
$j \in A_\ell$, $i \in B_\ell$ and thus the equation
\begin{align*}
    \sum_{h \in A_\ell, (h,i) \in Q} y_{hi} + \sum_{h \in A_\ell, (i,h) \in Q} y_{ih}   &=\; x_{i} &&  \tag{$****$} \label{eqn:rhsi_proof3}
\end{align*}
exists, has $y_{ij}$ on its left hand side, and thus establishes $y_{ij} \le x_i$.

\newpage

To show that $y_{ij} \ge x_i + x_j - 1$, consider equation~(\ref{eqn:rhsj_proof3}) in
combination with its original counterpart $\sum_{h \in A_k} x_h = 1$.
For any $y_{hj}$ (or $y_{jh}$) in~(\ref{eqn:rhsj_proof3}), the conditions~\ref{cond:c1}
and~\ref{cond:c2} assure that there is an equation establishing $y_{hj} \le x_h$ ($y_{jh} \le x_h$).
Thus we have
\begin{align}
    \sum_{h \in A_k, (h,j) \in Q, h \neq i} y_{hj} + \sum_{h \in A_k, (j,h) \in Q, h \neq i} y_{jh}   &\le\; \sum_{h \in A_k, h \neq i} x_{h} = 1 - x_i && \nonumber 
\end{align}

Applying this bound within equation~(\ref{eqn:rhsj_proof3}), we obtain:
$$y_{ij} + \underbrace{\sum_{h \in A_k, (h,j) \in Q, h \neq i} y_{hj} + \sum_{h \in A_k, (j,h) \in Q, h \neq i} y_{jh}}_{\le 1 - x_i}  = x_j\; \Leftrightarrow y_{ij} \ge x_i + x_j - 1$$
    \vspace{-0.2cm}
\end{proof}

If all the right hand sides of the original equations are equal to two
and the products to be induced are exactly those given by $A_k \times A_k$
for all $k \in K$, we obtain another important special case where the equations
induced by conditions~\ref{cond:c1} and~\ref{cond:c2} imply the
inequalities~(\ref{pck:stdlin1}), (\ref{pck:stdlin2}), and~(\ref{pck:stdlin3})
also for fractional solutions. We will see an example application where
this case occurs in practice in Sect.~\ref{ss:QTSP}.

\begin{theorem} \label{thm:tsp}
If, for all $k \in K$, (i) $a^k_i = 1$ for all $i \in A_k$, (ii) $b^k = 2$, (iii) $B_k = A_k$, and
the Conditions~\ref{cond:c1} and~\ref{cond:c2} are satisfied, then there is a compact linearization
such that, for any solution $x \in [0,1]^N$, the inequalities
$y_{ij} \le x_i$, $y_{ij} \le x_j$ and
$y_{ij} \ge x_i + x_j - 1$ hold for all $(i,j) \in Q$, $i \neq j$.
\end{theorem}
\begin{proof}
Due to (iii), the induced equations~(\ref{cmp:eqn1}) look like:
\begin{align}
    y_{jj} + \sum_{h \in A_k, h < j} y_{hj}\ + \sum_{h \in A_k, j < h} y_{jh}    &=\; 2 x_{j} && \mbox{for all } j \in A_k,\; \mbox{for all } k \in K \nonumber 
\end{align}

Since $y_{jj}$ shall take on the same value as $x_j$, we may eliminate $y_{jj}$ on
the left and once subtract $x_j$ on the right. We obtain:
\begin{align}
   \sum_{h \in A_k, h < j} y_{hj}\ + \sum_{h \in A_k, j < h} y_{jh}    &=\; x_{j} && \mbox{for all } j \in A_k,\; \mbox{for all } k \in K \label{cmp:sp2_2}
\end{align}

These equations establish inequalities~(\ref{pck:stdlin1}) and (\ref{pck:stdlin2}) for all $y_{ij}$, $(i,j) \in Q$, $i \neq j$.
Combining them with the original equations $\sum_{j \in A_k} x_j = 2$ yields:
\begin{align}
    2 = \sum_{j \in A_k} x_j = \sum_{j \in A_k} \big( \sum_{h \in A_k, h < j} y_{hj}\ + \sum_{h \in A_k, j < h} y_{jh} \big) = 2 * \sum_{j \in A_k} \sum_{h \in A_k, h < j} y_{hj} \nonumber 
\end{align}

It follows that $\sum_{j \in A_k} \sum_{h \in A_k, h < j} y_{hj} = 1$ holds even when $x$ is fractional.
Since we have for all $y_{ij}$, $(i,j) \in Q$, $i \neq j$ that
$\{i,j\} \subseteq A_k$ and, except $y_{ij}$, no two variables in the equations~(\ref{cmp:sp2_2}) expressed for $i$ and for $j$ coincide, we conclude:
\begin{align*}
    x_i + x_j &=  \sum_{h \in A_k, i < h} y_{ih}\ + \sum_{h \in A_k, h < i} y_{hi} + \sum_{h \in A_k, j < h} y_{jh}\ + \sum_{h \in A_k, h < j} y_{hj}\\
              &=  y_{ij} + \underbrace{\sum_{h \in A_k, i < h \neq j} y_{ih}\ + \sum_{h \in A_k, j \neq h < i} y_{hi} + \sum_{h \in A_k, j < h} y_{jh}\ + \sum_{h \in A_k, h < j} y_{hj}}_{\le 1}\\
              &\le\; y_{ij} + 1
\end{align*}
\end{proof}

\newpage

For the general setting of linear equations, we shall now
clarify on how to obtain a set $Q \supseteq P$
as compact as possible and how to induce also a minimum
number of additional linearization equations by minimizing
$\sum_{k\in K} |B_k|$. This works in the same fashion as
presented in~\citet{Mallach2017} for the case of
assignment constraints.
While the approach possibly involves creating more linearization
variables as originally demanded by the set $P$, the number of
equations will typically be considerably smaller than it would
be with the \lq standard\rq\ approach as is discussed
by~\citet{Liberti2007}. Depending on $P$ and the overlap among
the sets $A_k$, $k \in K$, the conditions~\ref{cond:c1}
and~\ref{cond:c2} impose a certain minimum on the sum of
cardinalities $\sum_{k \in K} |B_k|$ and hence on the number of equations.
Different solutions achieving this minimum may lead to different
cardinalities of $|Q|$. In general, it is even possible that a
minimum $|Q|$ can only be achieved by adding more than a minimum
number of equations. But, in practice, this is seldom the case as
adding equations rather induces more variables.

A consistent linearization of minimum
size can be computed by solving the following mixed-integer
program that reduces to a linear program with
totally unimodular constraint matrix whenever
$A_k \cap A_\ell = \emptyset$, for all $k,\ell \in K$, $\ell \neq k$.
\begin{align}
    \IPmin  &  w_{eqn} \bigg( \sum_{k \in K} \sum_{1 \le i \le n} z_{ik} \bigg) + w_{var} \bigg( \sum_{1 \le i \le n}\sum_{i \le j \le n} f_{ij} \bigg)    \span \span  \span\span \nonumber \\
\IPst   &  f_{ij}                                         &=\;   & 1          && \mbox{for all } (i, j) \in P  \label{minB:EinF0}  \\
        &  f_{ij}                                         &\ge\; & z_{jk}     && \mbox{for all } k \in K, i \in A_k, j \in N, i \le j \label{minB:EinF3a} \\
	&  f_{ji}                                         &\ge\; & z_{jk}     && \mbox{for all } k \in K, i \in A_k, j \in N, j < i \label{minB:EinF3b} \\
        &  \sum_{k: i \in A_k} z_{jk} \qquad\qquad\qquad  &\ge\; & f_{ij}     && \mbox{for all } 1 \le i \le j \le n \label{minB:EinF1} \\
	&  \sum_{k: j \in A_k} z_{ik}                     &\ge\; & f_{ij}     && \mbox{for all } 1 \le i \le j \le n \label{minB:EinF2} \\
	&  f_{ij}                                         &\in\; & [0,1]    && \mbox{for all } 1 \le i \le j \le n \nonumber \\
        &  z_{ik}                                         &\in\; & \{0,1\}    && \mbox{for all } k \in K, 1 \le i \le n  \nonumber
\end{align}

The formulation involves binary variables $z_{ik}$ to be equal to $1$ if $i \in B_k$
and equal to zero otherwise. Further, to account for whether $(i,j) \in Q$, there
is a (continuous) variable $f_{ij}$ for all $1 \le i \le j \le n$ that will be
equal to $1$ in this case and $0$ otherwise.
The constraints \cref{minB:EinF0} fix those $f_{ij}$ to $1$ where the corresponding
pair $(i,j)$ is contained in $P$. Whenever some $j \in N$
is assigned to some set $B_k$, then we induce the corresponding products $(i,j) \in Q$
or $(j,i) \in Q$ for all $i \in A_k$ which is established by \cref{minB:EinF3a} and
\cref{minB:EinF3b}. Finally, if $(i,j) \in Q$, then we require
Conditions~\ref{cond:c1} and~\ref{cond:c2} to be satisfied, namely that there is
a $k \in K$ such that $i \in A_k$ and $j \in B_k$ \cref{minB:EinF1} and a
(possibly different) $k \in K$ such that $j \in A_k$ and $i \in B_k$ \cref{minB:EinF2}.
The weights $w_{var}$ and $w_{eqn}$ of the objective function allow to
find the best compromise between a minimum number of variables and a minimum number of
equations. A rational choice would be to set $w_{var} = 1$ and
$w_{eqn} > \max_{k \in K} |A_k|$. This results in a solution with a minimum number
of additional equations that, among these, also induces a minimal number of
additional variables.

In the mentioned case of disjoint supports, i.e.,
$A_k \cap A_\ell = \emptyset$, for all $k,\ell \in K$, ${\ell \neq k}$,
the unique minimum-cardinality sets $B_k$ satisfying conditions~\ref{cond:c1}
and~\ref{cond:c2} for all ${(i,j)\in Q}$ can as well be computed by a
combinatorial algorithm as presented in~\cite{Mallach2017}.

So while it is possible to carry out the linearization problem in a very
general form algorithmically (e.g.~as part of a preprocessing step of a solver),
the (most) compact linearization associated to a particular problem formulation
is typically \lq recognized\rq\ easily by hand as we will see exemplary
in the following section.

\section{Applications}\label{s:Apps}

In this section, we highlight some prominent quadratic combinatorial
optimization problems where linearizations found earlier appear as
special cases of the proposed compact linearization technique.

\subsection{Quadratic Assignment Problem}

As discussed by \citet{Liberti2007}, the linearization of
the Koopmans-Beckmann formulation for the quadratic assignment problem
as presented by~\citet{FriezeYadegar83} is exactly the one that results
when applying the compact linearization technique to it.

\subsection{Symmetric Quadratic Traveling Salesman Problem}\label{ss:QTSP}

In the form as discussed by \citet{FischerH13},
the symmetric quadratic traveling salesman problem asks for a tour $T$
in a complete undirected graph ${G=(V,E)}$ such that the objective
$\sum_{\{i,j,k\} \subseteq V, j \neq i < k \neq j } c_{ijk} x_{ij} x_{jk}$
(where $x_{ij} = 1$ if $\{i,j\} \in T$) is minimized.
Consider the following mixed-integer programming formulation for this problem
that is oriented at the integer programming formulation for the linear traveling
salesman problem by~\citet{Dantzig54}.
\begin{align}
\IPmin & \sum_{\{i,j,k\} \subseteq V, j \neq i < k \neq j } c_{ijk} y_{ijk} \span\span\span\span \nonumber \\
\IPst  & \sum_{\{i,j\} \in E} x_{ij}  &=\;   & 2     &&\mbox{for all}\; i \in V \label{qtsp:degree}\\
       & x(E(W))                                 &\le\; & |W|-1  &&\mbox{for all}\; W \subsetneq V,\; 2 \le |W| \le |V| - 2 \nonumber \\ 
       & y_{ijk}                                 &=\; & x_{ij} x_{jk} &&\mbox{for all}\; \{i,j,k\} \subseteq V, j \neq i < k \neq j \label{qtsp:lin}\\
       & x_{ij}                                 &\in\; & \{0,1\}     &&\mbox{for all}\; \{i,j\} \in E \nonumber
\end{align}

In the context of the compact linearization approach proposed, we
consider the linear equations~(\ref{qtsp:degree}) where we have
$K = V$, $A_k = \{ jk \mid j < k \mbox { and } \{j,k\} \in E \}$,
$a_i^k = 1$ for all $i \in A_k$ and $b^k = 2$ for all $k \in K$.
Since we are interested in the bilinear terms of the form as in~(\ref{qtsp:lin}),
i.e.~each pair of edges with common index $j$, we need to set $B_k = A_k$ for
all $k \in K$ in order to satisfy both conditions~\ref{cond:c1} and~\ref{cond:c2}
for each such pair.
We thus comply to the requirements of the special case addressed in Theorem~\ref{thm:tsp}
and obtain the equations:
\begin{align}
    \sum_{\{i,j\} \in E}  x_{ij} x_{jk}     &=\; 2 x_{jk} && \mbox{for all } \{j,k\} \in E,\; \mbox{for all } j \in V \nonumber 
\end{align}

After introducing linearization variables with indices ordered as desired, these are resolved as:
\begin{align}
    \sum_{\{i,j,k\} \subseteq V, j \neq i \le k \neq j} y_{ijk} &=\; 2 x_{jk} && \mbox{for all } \{j,k\} \in E,\; \mbox{for all } j \in V \nonumber 
\end{align}

Each of these equations induces one variable more than originally demanded which
is $y_{kjk}$ as the linearized substitute for the square term $x_{jk} x_{jk}$.
Since this product is zero if $x_{jk} = 0$ and equal to one if $x_{jk} = 1$,
we may safely subtract $y_{kjk}$ from the left and $x_{jk}$ from the right hand side
and obtain
\begin{align}
    \sum_{\{i,j,k\} \subseteq V, j \neq i < k \neq j} y_{ijk} &=\; x_{jk} && \mbox{for all } \{j,k\} \in E,\; \mbox{for all } j \in V \nonumber
\end{align}
which are exactly the linearization constraints as proposed by~\citet{FischerH13}.

\section{Conclusion}\label{s:Concl}

We showed that the compact linearization approach can be
applied not only to binary quadratic problems with assignment
constraints, but to those with arbitrary linear equations with
positive coefficients. We discussed two special cases under which
the continuous relaxation of the obtained compactly linearized
problem formulation is provably as least as strong as the one
obtained with an ordinary linearization. This is true for the
original case of assignment constraints, but also for another
setting where the right hand sides of the equations are equal to two.
Finally, we highlighted previously found linearizations that appear
as special cases of the proposed compact linearization technique, namely
the one by Frieze and Yadegar found for the quadratic assignment problem,
and the one by Fischer and Helmberg for the symmetric quadratic traveling
salesman problem.

\section*{Acknowledgments}

I would like to thank Anja Fischer for bringing up the question
whether the original approach for assignment constraints might
be possibly generalized to arbitrary right hand sides and
pointing me to her research about the symmetric quadratic traveling
salesman problem.

\bibliographystyle{spbasic}
\bibliography{bibfile}

\end{document}